\documentclass{amsart}

\usepackage{graphicx}
\usepackage{float}
\usepackage{hyperref}
\usepackage{amssymb}
\usepackage{amsmath}
\usepackage{bbm,bm}
\usepackage[all]{xy}

\def\Fs{\mathfrak S}
\def\Fm{\mathcal M}
\def\Fy{\mathcal Y}
\def\Fq{\mathcal Q}
\def\Fck{\mathcal C \mathcal K}
\def\Fssym{\mathfrak SSym}
\def\FC{\mathfrak C}
\def\Fpsym{\mathcal PSym}
\def\Fysym{\mathcal YSym}

\newcommand{\FsetS}{{\sf S}}
\newcommand{\FsetT}{{\sf T}}
\newcommand{\FsetP}{{\sf P}}
\newtheorem{theorem}{Theorem}[section]
\newtheorem{defi}[theorem]{Definition}
\newtheorem{exam}[theorem]{Example}
\newtheorem{rema}[theorem]{Remark}
\newenvironment{definition}[1][]{\rm\begin{defi}[#1]\rm}{\end{defi}}
\newenvironment{example}[1][]{\rm\begin{exam}[#1]\rm}{\end{exam}}
\newenvironment{remark}[1][]{\rm\begin{rema}[#1]\rm}{\end{rema}}

\newcommand{\FDdeg}[1]{(#1)}

\newcommand{\FPlaceholder}%
    {\raisebox{.09ex}{\footnotesize $\bullet$}}

\def\Fbb{\Fboldbullet}
    \def\Fboldbullet{{\!\mbox{\LARGE$\mathbf\cdot$}}}

\def\Fpsplit{\stackrel{\curlyvee}{\to}}

\def\Ftreeseparator{\bm\cdot}
\def\Ftreeseparatorinline{\bm\cdot}
\newcommand{\Fzcompose}[2]{\frac{{#1}}{#2}}
\newcommand{\Fcompose}[3]{\frac{{#1{\Ftreeseparator}\dotsb{\Ftreeseparator}#2}}{#3}}
\newcommand{\Flrcompose}[5]{\frac{{#1{\Ftreeseparator}#2{\Ftreeseparator}\dotsb{\Ftreeseparator}#3{\Ftreeseparator}#4}}{#5}}
\newcommand{\Fcomposeinline}[3]{{({#1{\Ftreeseparatorinline}\dotsb{\Ftreeseparatorinline}#2}){\,\circ\,}#3}}

\begin{document}


\title{Extending the Tamari lattice to some compositions of species}
\author{Stefan Forcey} \address[S. Forcey]{
    Department of Mathematics\\
    The University of Akron\\
    Akron, OH 44325-4002
    }
    \email{sf34@uakron.edu}
    \urladdr{http://www.math.uakron.edu/\~{}sf34/}
%
%

\begin{abstract}
An extension of the Tamari lattice to the multiplihedra is discussed,
  along with projections to the composihedra and the Boolean lattice.
   The multiplihedra and composihedra are sequences of polytopes that  arose in algebraic topology
    and category theory. Here we describe them in terms of the composition of combinatorial species.
    We define lattice structures on their vertices, indexed by painted trees,
    which are extensions of the Tamari lattice and projections of the weak order on
    the permutations. The projections from the weak order to the
    Tamari lattice and the Boolean lattice are shown to be different
     from the classical ones. We review how lattice structures often
     interact with the Hopf algebra structures, following Aguiar and Sottile who
      discovered the applications of M\"obius inversion on the Tamari lattice to the Loday-Ronco Hopf algebra.
\end{abstract}
\maketitle
\section{Introduction}
We will be looking at the following spectrum of polytopes:
\[
    \includegraphics[width = \textwidth]{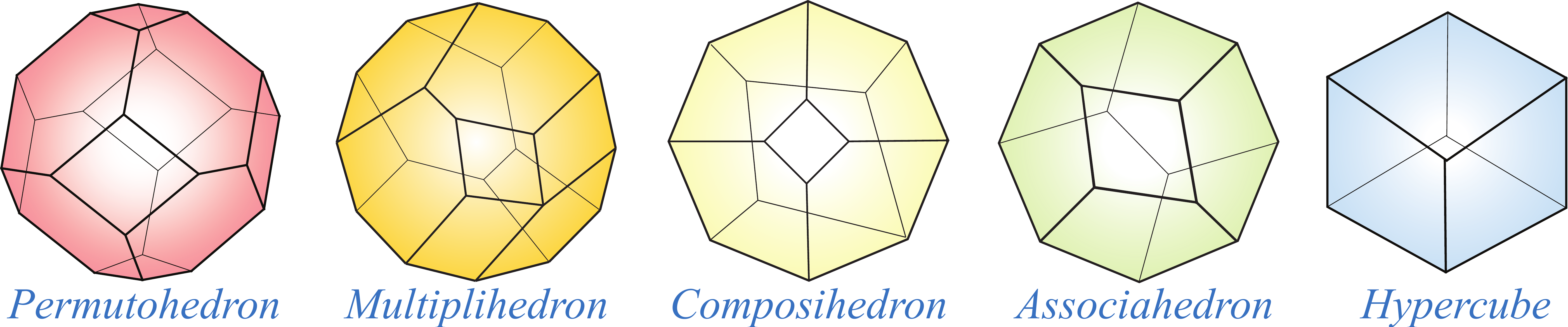}
\]

Here is a \emph{planar rooted binary tree}, often called a binary
tree:
\[
    \includegraphics[width = 3in]{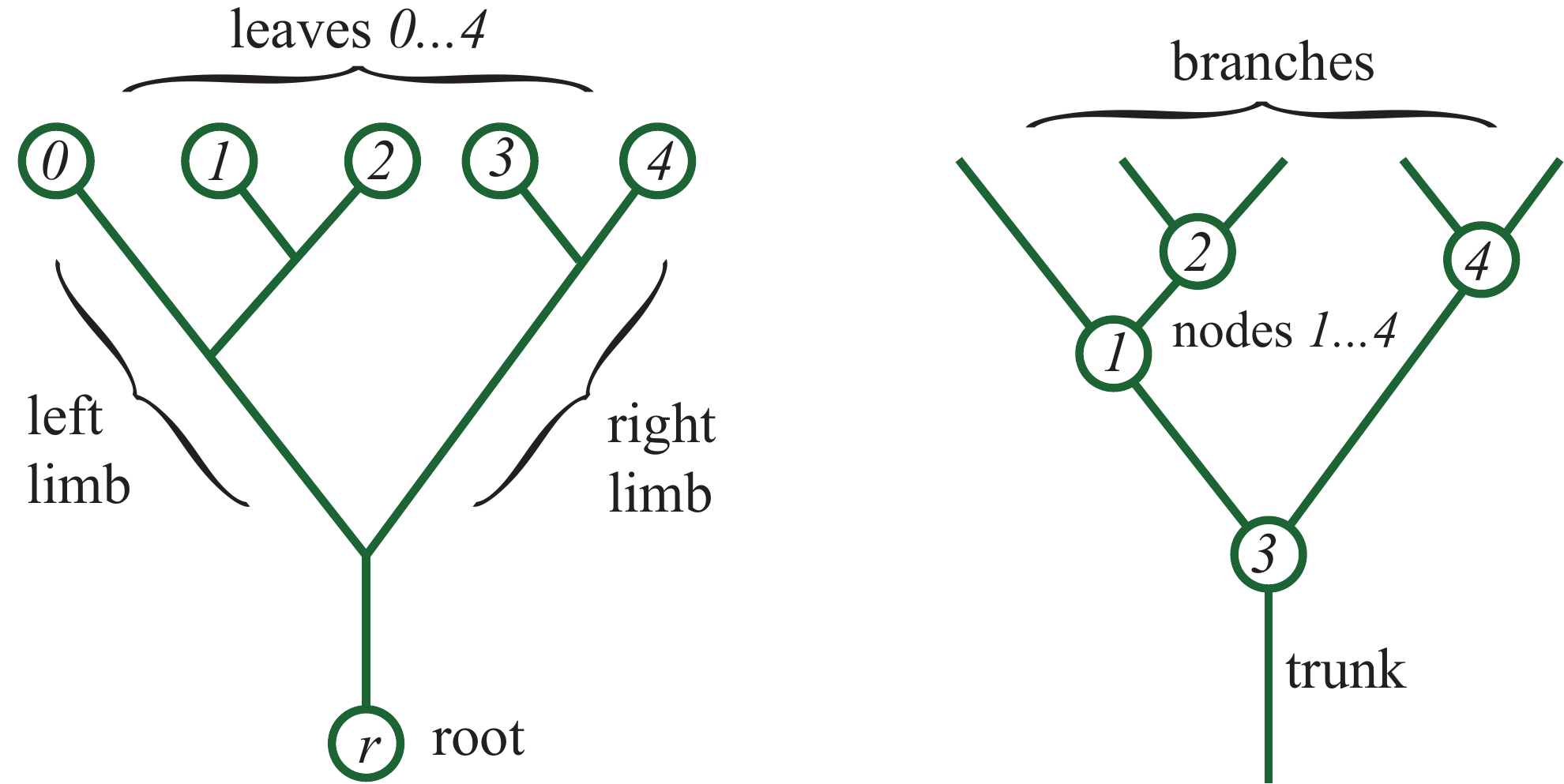}
\]
We only label the leaves and interior nodes (branch points) with
their left-to-right ordering when necessary. The branches are the
edges with a leaf. The nodes are also partially ordered by their
proximity to the root, which is maximal; in the picture node $3 > 1
> 2.$ The set of planar rooted binary trees with $n$ nodes and $n+1$
leaves is denoted $\Fy_n.$ 

\subsubsection{Notation}  We will choose from the current
rather prolix notation used for the classical polytopes and
lattices, and try to decrease the proliferation of symbols by
referring to a polytope and its associated orders by a common name.
The context will determine whether we are focused on the face
structure, the vertices alone, or the 1-skeleton. Since the vertices
are used more often, we let that be the default meaning, and add
more specifics if necessary. For instance if we wish to refer to the
general planar trees that index the faces of the associahedron,
we'll make that clear.

 The Tamari lattice  is denoted  either
$\Fy_n$ or ${\mathbb T}_n.$ The set of painted trees with $n$ nodes and $n+1$
leaves is denoted $\Fm_n$ and the lattice structure on that set is denoted
$\Fm_n$ as well. The set of binary trees with leaves weighted by positive
integers summing to $n+1$ is denoted $\Fck_n$ and the lattice structure we
define on that set is denoted $\Fck_n$ as well. Let $[n] = \{1,\dots,n\}.$ The
Boolean lattice of subsets of $[n]$ is denoted $\Fq_n.$ The lattice of weakly
ordered permutations on $[n]$ , described below, is denoted $\Fs_n.$ We also
continue this abusive notation by using the same symbols to denote the
polytopes whose vertices are indexed by the indicated set. Thus the $(n-1)$
-dimensional associahedron is denoted $\Fy_n,$ which corresponds to the
notation ${\mathcal{K}}_n$ in \cite{F_LodRon:1998} or ${\mathcal{K}}(n+1)$ in
\cite{F_forcey1} or even ${K}_{n+1}$ in Stasheff's original notation. The
$(n-1)$-dimensional permutohedron, multiplihedron, composihedron and hypercube
are denoted $\Fs_n, ~\Fm_n,  ~\Fck_n,$ and $\Fq_n$ respectively. Rather than a
subscript $n$, we sometimes use a placeholder $\bullet$ to refer to the entire
sequence at once.

 The 1-skeleta of the families of polytopes
$\Fs_\Fbb,\Fm_\Fbb$, $\Fy_\Fbb$ and $\Fq_\Fbb$ are Hasse diagrams of posets.
For the permutahedron $\Fs_n$, the corresponding poset is the (left) \emph{weak
order}, which we describe in terms of permutations. A cover in the weak order
has the form $w\lessdot (k,k{+}1) w$, where $k$ precedes $k{+}1$ among the
values of $w$. Figure~\ref{F_f:bigthree} displays the weak order on $\Fs_4$,
the Tamari order on $\Fy_4$ and the Boolean lattice $\Fq_3.$
\begin{figure}[hbt]
\[\includegraphics[width = \textwidth]{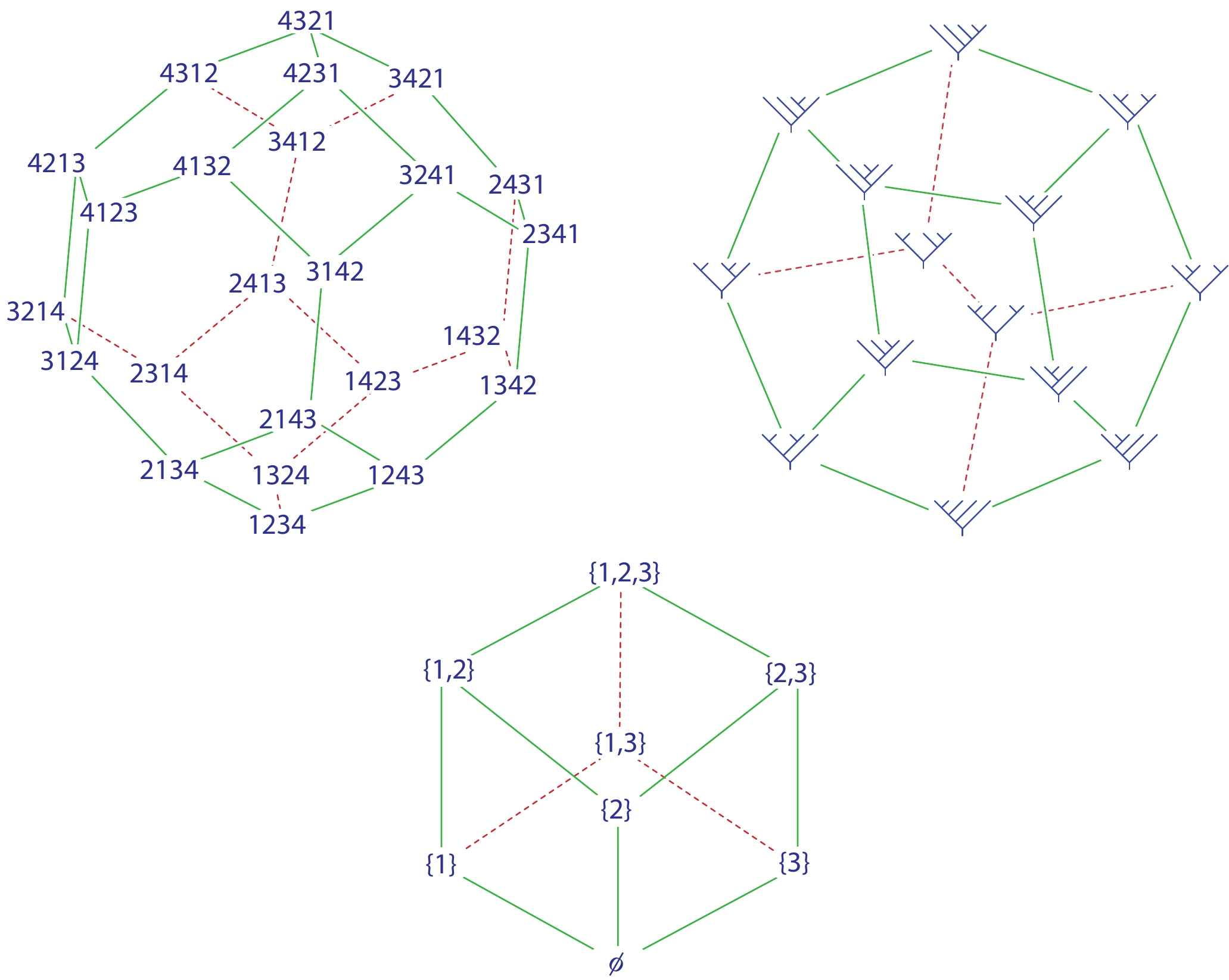}
\]
\caption{Hasse diagrams of three classical lattices in this paper: Weak order
on $\Fs_4$, Tamari lattice $\Fy_4$ and Boolean lattice $\Fq_3$. }
\label{F_f:bigthree}
\end{figure} 

\subsection{Species}
A combinatorial \emph{species} of sets is an endofunctor of Finite
Sets with bijections. 
\begin{example}
 The species ${\mathcal L}$ of lists takes a set to
linear orders of that set.
  $${\mathcal L}(\{a,d,h\}) =
  \{a<d<h,~{a<h<d,}~{h<a<d,}~{h<d<a,}~{d<a<h,}~{d<h<a}~
  \}$$
\end{example}

\begin{example}

The species ${\Fy}$ of binary trees takes a set to
  trees with labeled leaves.

$$
\hspace{-.75in}\includegraphics[width=3in]{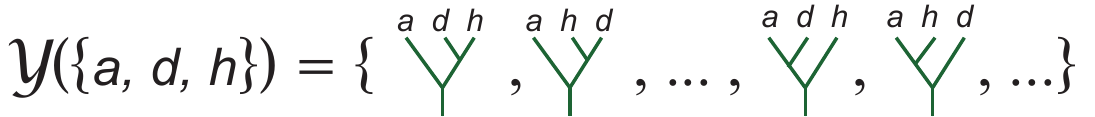}
$$
\end{example}

We define the composition of two species following Joyal in \cite{Joyal}:

  $$({\mathcal G} \circ {\mathcal H})(U) = \bigsqcup_{\pi} {\mathcal G}(\pi)\times\prod_{U_i \in \pi } {\mathcal H}(U_i) $$

  where the union is over partitions of $U$ into any number of nonempty disjoint parts.

  $$\pi= \{U_1,U_2,\dots,U_n\}\text{ such that } U_1 \sqcup \dots \sqcup U_n =
  U.$$

This formula also appears to be known as the cumulant formula, the
moment sequence of a random variable, and the domain for operad
composition:
$$\gamma: \mathcal{F} \circ \mathcal{F} \to \mathcal{F}.$$

\section{Several flavors of trees}
\subsection{Ordered, Bi-leveled and Painted trees}\label{F_sec: bi-leveled trees}
Many variations of the idea of the binary tree have proven useful in
applications to algebra and topology. Each variation we mention can
have its leaves labeled, providing an example of a set species.

 An ordered tree (sometimes
called leveled) has a vertical ordering of the $n$ nodes as well as
horizontal. This allows a well-known bijection between the ordered
trees with $n$ nodes and the permutations $\Fs_n.$ We will call this
bijection $bij_1$. This bijection and all the other maps we will
discuss are demonstrated in Figure~\ref{F_f:mapmap}.

As defined in 2.1 of \cite{F_FLS:2010},
 a \emph{bi-leveled tree}
$(t;\FsetT)$ is a planar binary tree $t\in \Fy_n$ together with an
(upper) order ideal $\FsetT$ of its node poset, where $\FsetT$
contains the leftmost node of $t$ as a minimal element. (Recall that
an upper order ideal is a sub-poset such that $x > y\in \FsetT $
implies $x \in \FsetT.$) We draw the underlying tree $t$ and circle
the nodes in $\FsetT$. By the condition on $\FsetT$, all nodes along
the leftmost branch are circled and none are circled above the
leftmost node.

Saneblidze and Umble~\cite{F_SanUmb:2004} introduced bi-leveled
trees in terms of equivalence classes on ordered trees. They
describe a cellular projection from the permutahedra to Stasheff's
multiplihedra $\Fm_\Fbb$, with the bi-leveled trees on $n$ nodes
indexing the vertices $\Fm_n$.

\begin{definition} We denote Saneblidze and Umble's map as $\beta: \Fs_n\to \Fm_n,$
as in \cite{F_FLS:2010}, and describe it as the map which first
circles all the nodes vertically ordered below and including the
leftmost node, and then forgets the vertical ordering of the nodes.
\end{definition}

Numbering the nodes in a tree $t\in\Fy_n$ $1,\dotsc,n$ from left to
right, $\FsetT$ becomes a subset of $\{1,\dotsc,n\}$.
\begin{definition}
The partial order on $\Fm_n$ is defined by $(s;\FsetS)\leq(t;\FsetT)$ if $s\leq
t$ in $\Fy_n$ and $\FsetT\subseteq \FsetS$.
\end{definition}
\begin{theorem}
The poset of bi-leveled trees is a lattice.
\end{theorem}
\begin{proof} The unique supremum of two bi-leveled trees $(t;\FsetT) $ and $(s;\FsetS) $ is found by first
taking their unique supremum $\sup\{t,s\}$ in the Tamari lattice,
and then circling as many nodes of $\sup\{t,s\}$ in the intersection
$\FsetT \cap \FsetS$ as are allowed by the upper order ideal
condition. That is, the circled nodes of the join comprise the
largest order ideal of nodes of $\sup\{t,s\}$ that is contained in
$\FsetT \cap \FsetS.$ The unique infimum is found  by taking the
infimum of the two trees in the Tamari lattice, the union of the two
order ideals, and adding to the latter any nodes necessary to make
that union an order ideal in the node poset of $\inf\{s,t\}$. That
is, the meet is given by $(\inf\{t,s\};\FsetT \cup\FsetS \cup\{x~|~x
>y \in \FsetT \cup\FsetS\}).$
\end{proof}

The Hasse diagrams of the posets $\Fm_n$ are $1$-skeleta for the
multiplihedra. The Hasse diagram of $\Fm_4$ appears in
Figure~\ref{F_fig: M4}. Stasheff used a different type of tree for
the vertices of $\Fm_\Fbb$. A \emph{painted binary tree} is a planar
binary tree $t$, together with a (possibly empty) upper order ideal
of the node poset of $t$. (Recall the root node is maximal.) We
indicate this ideal  by painting part of a representation of $t$.
For clarity, we stop our painting in the middle of edges (not
precisely at nodes). Here are a few simple examples,
\[
    \includegraphics{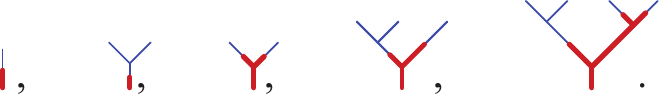}
\]

An \emph{$A_n$-space} is a topological $H$-space with a weakly
associative multiplication of points \cite{F_Sta:1963}. Maps between
$A_{\infty}$ spaces are only required to preserve the $A_{\infty}$
structure up to homotopy.
 Stasheff~\cite{F_Sta:1970}
described these maps combinatorially using cell complexes called
multiplihedra, while Boardman and Vogt~\cite{F_BoaVog:1973}  used
spaces of painted trees. Both the spaces of trees and the cell
complexes are homeomorphic to convex polytope realizations of the
multiplihedra as shown in \cite{F_forcey1}.

If $f\colon ( X,\FPlaceholder) \to (Y,\ast)$ is an $A_{\infty}$-map
homotopy $H$-spaces, then the different ways to multiply and map $n$
points of $X$ are naturally represented by a painted tree, as
follows. Nodes not painted correspond to multiplications in $X$,
painted nodes correspond to multiplications in $Y$, and the
beginning of the painting (along the edges) indicates the moment $f$
is applied to a given point in $X$. (Weak associativity of $X$ and
$Y$ justifies the use of planar binary trees, which represent the
distinct associations on a set of inputs.) See Figure \ref{F_fig:
maps to painted}.
\begin{figure}[hbt!]
\[
    f({a})\,{\bm\ast}\,\bigl(f({b\,\FPlaceholder\,c})\,
     {\bm\ast}\,f({d})\bigr)
    \ \longleftrightarrow \
    {\raisebox{-.5\height}{\includegraphics{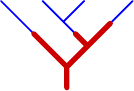}}}
\]
\caption{$A_n$-maps between $H$-spaces $({ X},{ \bullet})
\stackrel{f}{\longrightarrow} ({ Y},{ \bm\ast})$ are painted trees.}
\label{F_fig: maps to painted}
\end{figure}

Figure~\ref{F_fig: M4} shows two versions of the three-dimensional
multiplihedron as Hasse diagrams.

Bi-leveled trees having $n{+}1$ internal nodes are in bijection with
painted trees having $n$ internal nodes, the bijection being given
by pruning: Remove the leftmost branch (and hence, node) from a
bi-leveled tree to get a tree whose order ideal is the order ideal
of the bi-leveled tree, minus the leftmost node. We refer to this as
$bij_2$. This mapping and its inverse are illustrated in
Figure~\ref{F_fig: painted to bi-leveled}. The composition of
$bij_2$ with $\beta$ is just called $\beta.$  (We will often use
these bijections as identities.)
\begin{figure}[htb]
\[
  \includegraphics{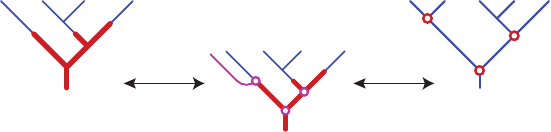}
\]
\caption{Painted trees correspond to bi-leveled trees.
}\label{F_fig: painted to bi-leveled}
\end{figure}
\begin{figure}[hbt!]
\[
    {\includegraphics[width = \textwidth]{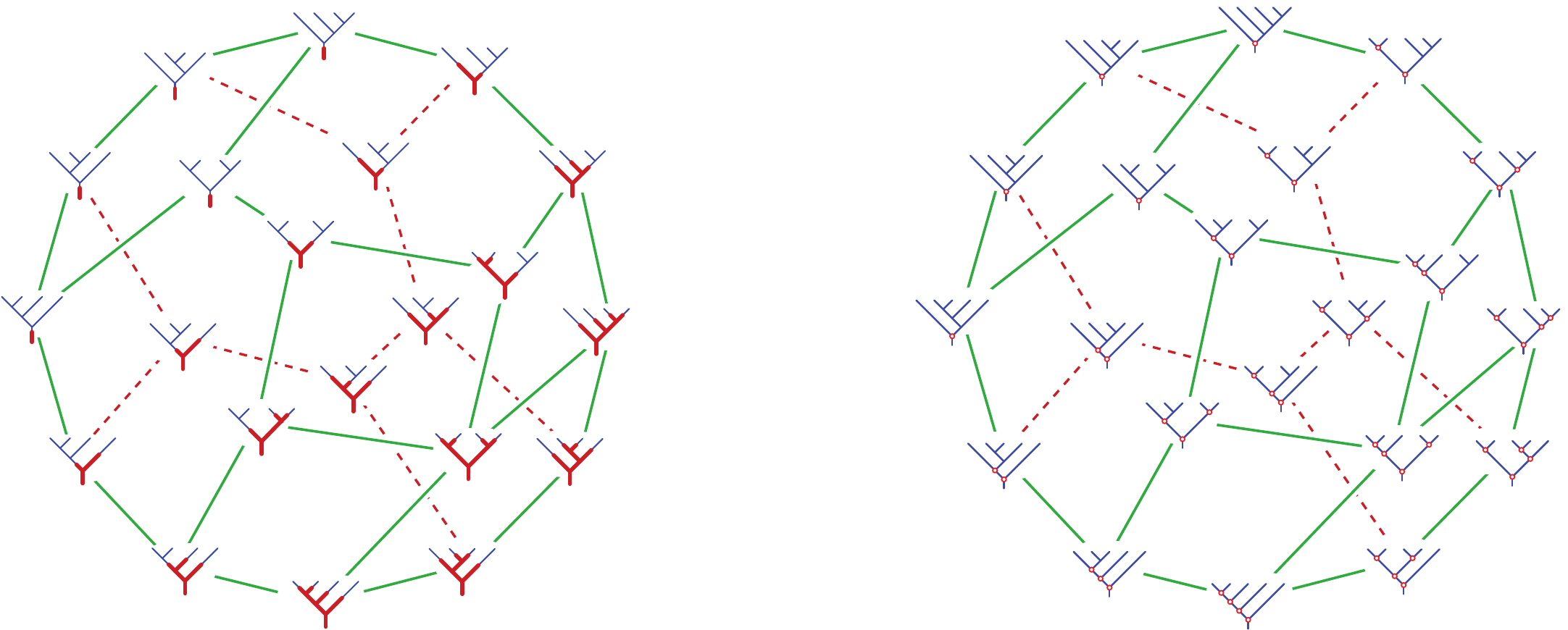}}
\]
\caption{Two Hasse diagrams of the multiplihedra lattice $\Fm_4,$
labeled with painted and bi-leveled trees. } \label{F_fig: M4}
\end{figure}
\begin{remark}
If the leaves of a painted binary tree are labeled by the elements
of a set, it is recognizable as a structure in a certain
combinatorial species: the self-composition $\Fy \circ \Fy$ of
binary trees. The structure types of this species are the
(unlabeled) binary painted trees themselves. Forgetting the painting
in any painted tree is precisely the composition in the operad of
binary trees.
\end{remark}

Forgetting the levels in a bi-leveled tree (removing the circles)
gives a different (from the one just remarked on) projection to
binary trees. We denote this by $\phi: \Fm_n\to\Fy_n$ as in
\cite{F_FLS:2010}. Now the composition $\phi\circ\beta$ gives the
Tonks projection from $\Fs_n$ to $ Fy_n,$
 denoted respectively by
$\Theta$ in \cite{F_Ton:1997}, by $\tau$ in \cite{F_FLS:2010} and by
$\Psi$ in \cite{F_LodRon:1998}.

In \cite{F_LodRon:1998} Loday and Ronco define a poset map from $\Fy_n$ to
$\Fq_n.$ (They call it $\phi,$ here we denote it $\hat{\phi}$ to avoid
duplicate naming.) This map takes a tree and gives a vertex of the hypercube
$[-1,1]^n$ by assigning either $+1$ or $-1$ to each of the branches not on a
limb. Each branch is assigned its slope, where the tree is drawn with 45 degree
angles. Further, they use  a bijection (we call it $bij_4$) from these vertices
to elements of the boolean lattice $\Fq_n$ defined by including the elements
$i\in[n-1]$ which correspond to the coordinates $x_i = -1.$ The composition of
$\hat{\phi} , \phi$  and $\beta$ gives the descents of the permutation.
\begin{figure}[hbt!]
\[
    {\includegraphics[width = \textwidth]{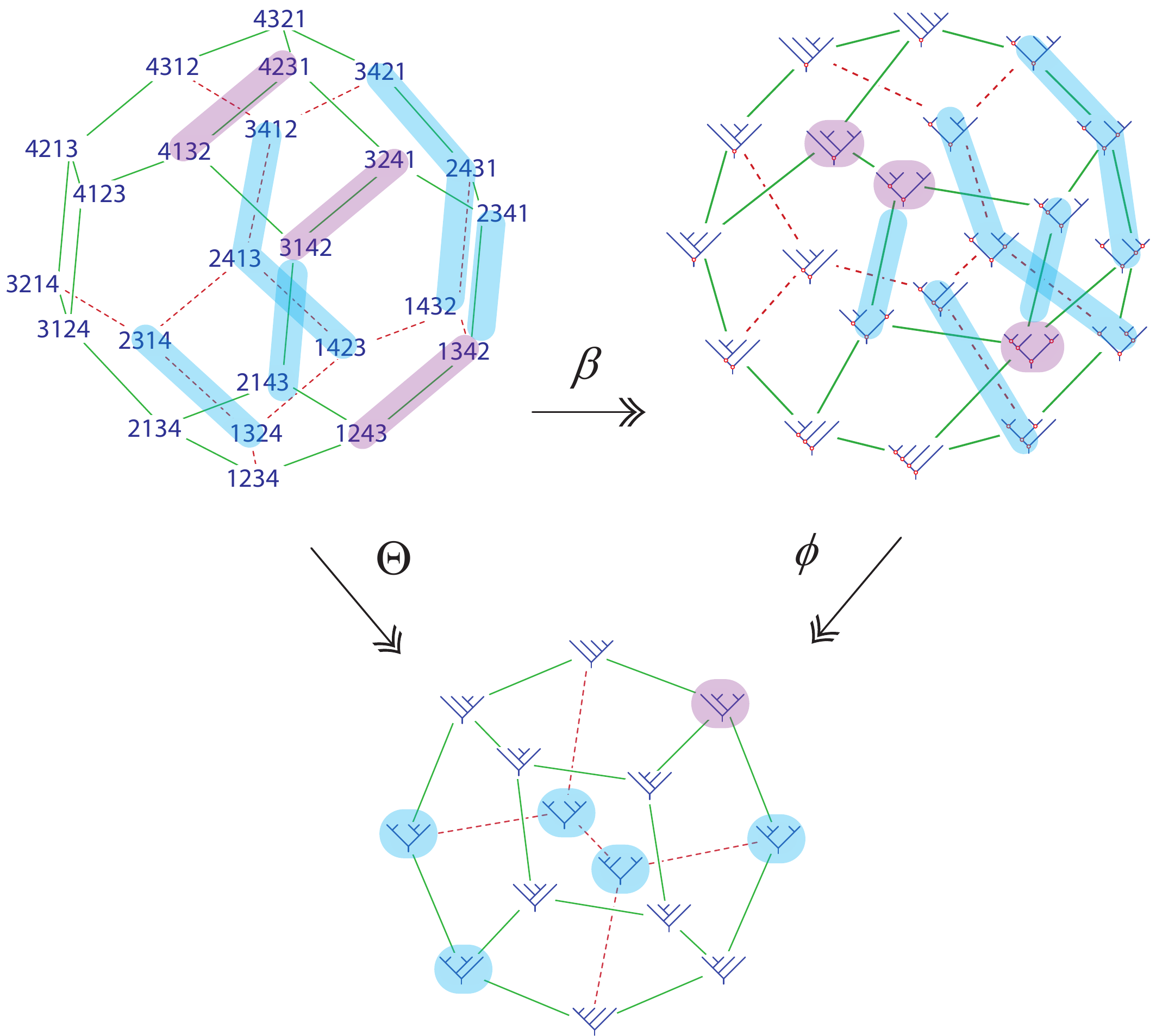}}
\]
\caption{The maps $\Theta =  \phi \circ \beta  $ shown with
retracted intervals shaded.} \label{F_fig:betaphi}
\end{figure}
\subsection{Trees with corrollas}

We will use the term corolla to describe a rooted tree with one
interior node and $n+1$ leaves. In a forest of corollas attached to
a binary tree, each corolla may be replaced by a positive
\emph{weight} counting the number of leaves in the corolla.
(Alternately, as in \cite{F_FLS3}, the corollas may be replaced by
combs.) These all give \emph{weighted trees}.
 \begin{equation}\label{F_Eq:composite_trees}
     \includegraphics{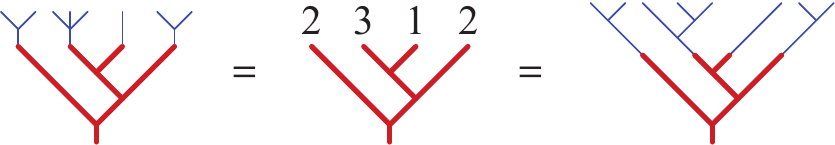}
 \end{equation}
%
\begin{remark}
By labeling leaves of a comb by the elements of a set, we define a species
called $\FC,$ which is in fact isomorphic to the species of lists. Labeled
weighted trees (as combs grafted to a tree) are recognizable as the  structures
in the species composition $\Fy \circ \FC$.
\end{remark}

Let $\Fck_n$ denote the weighted trees with weights summing to
$n{+}2$. These index the vertices of the $n$-dimensional
\emph{composihedron}, $\Fck(n{+}1)$~\cite{F_forcey2}. This sequence
of polytopes parameterizes homotopy maps from strictly associative
$H$-spaces to $A_{\infty}$-spaces. 

If we use right combs instead of corollas as the weights on our
weighted trees, then the same relations as in $\Fm_n$ give the
weighted trees a lattice structure. The joins and meets are found as
for painted trees, with the final step of combing the unpainted
subtrees. Figure~\ref{F_F:comp} gives two pictures of the
composihedron $\Fck_4$.
\begin{figure}[hbt]
\[
  {\includegraphics[width = \textwidth]{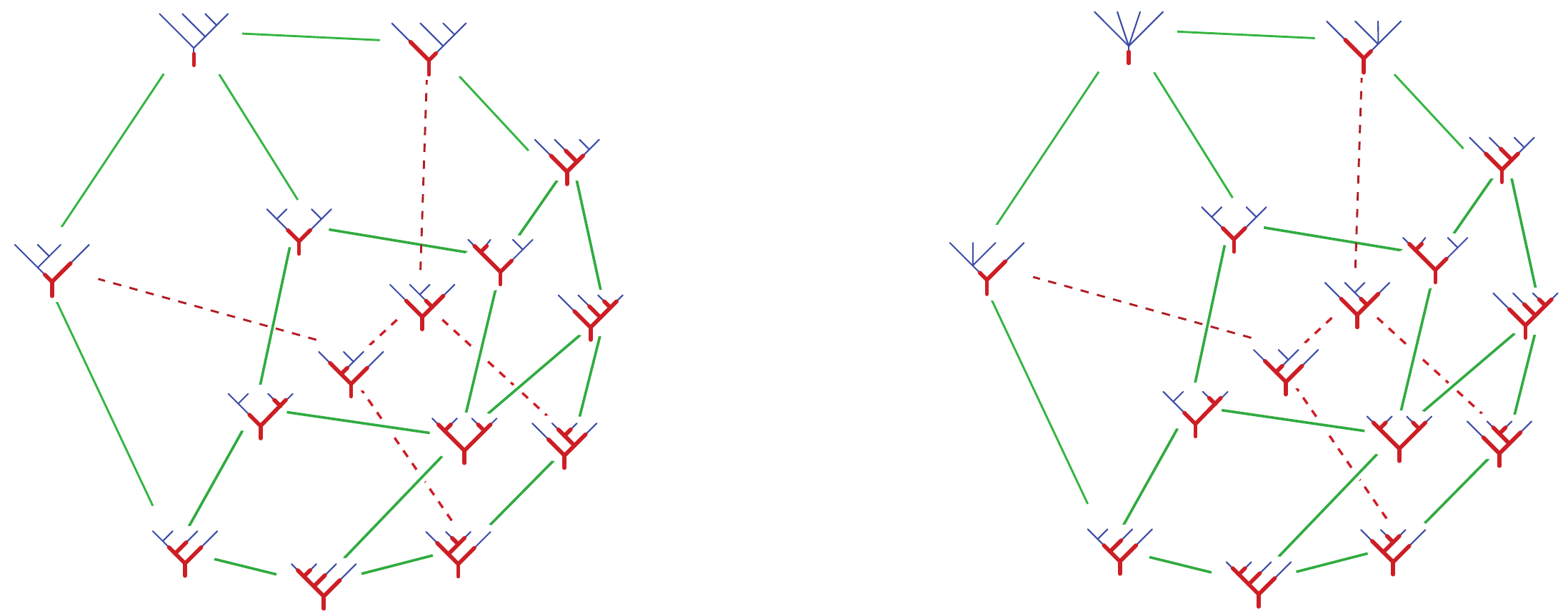}}
\]
\caption{The one-skeleton of the three-dimensional composihedron, as
a Hasse diagram labeled by two representations of weighted trees.}
\label{F_F:comp}
\end{figure}
The 2 and 3-dimensional composihedra $\Fck(n)$ also appear as the
commuting diagrams in enriched bicategories~\cite{F_forcey2}. As a
special case of enriched bicategories, these diagrams appear in the
definition of pseudomonoids~\cite[Appendix~C]{F_AguMah:2010}.

On the other hand, attaching a forest of binary trees to a single
corolla is really just a way of picturing an ordered forest of
binary trees, listed left to right. There is a well known bijection
from the set of ordered forests with $n+1$ total leaves to $\Fy_n.$
We call this bijection $bij_3$. It is described by taking the $k$
trees of the forest in left-to-right order and attaching them to a
single limb, which will be the new left limb. Thus we can recognize
this set as another version of $\Fy_n.$

Finally we consider trees with $n$ interior nodes obtained by
grafting a forest of combs to the leaves of a comb (which is
painted). Analogous to~\ref{F_Eq:composite_trees}, these are
weighted combs (or corollas). As these are in bijection with
number-theoretic compositions of $n{+}1$, we refer to them as
\emph{composition
  trees}.
\[
   \includegraphics{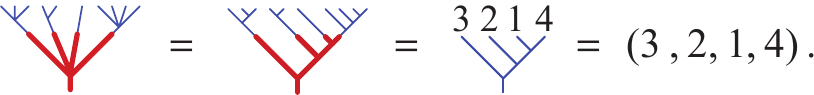}
\]

\begin{remark}
Leaf-labeled composition trees (where we are labeling the leaves of the forest
of combs grafted to a comb) are recognizable as the structures in the species
composition $\FC \circ \FC$.
\end{remark}

 In the next section we will describe maps from
the multiplihedra to the hypercubes, but first we note that we will
use a different bijection from the set of composition trees to
$\Fq_n.$ A composition tree is associated by bijection $bij_5$ with
the set of vertices that are unpainted.

\section{Interval retracts}

In \cite{F_FLS:2010} it is shown that there exists a section of the
projection $\beta:\Fs_n \to \Fm_n$ which demonstrates $\beta$ to be
an interval retract. We review that definition here. Recall that an
interval $[a,b]$ of a poset $\FsetP$ is a sub-poset given by $\{x
~|~ a\le x \le b \in \FsetP.\}$

 A surjective poset map $f\colon
P\to Q$ from a finite lattice $P$ is an \emph{interval retract} if
the fibers of $f$ are intervals and if $f$ admits an
order-preserving section $g\colon Q\to P$ with $f\circ g=\mbox{id}$.
Also in \cite{F_FLS:2010} there is proven a useful relation between
the M\"obius functions of $\Fs_\Fbb$ and $\Fm_\Fbb,$ which is in
fact established there in a general form.
\begin{theorem}\label{F_thm: subGalois} \cite{F_FLS:2010}
Let the poset map $f\colon P\to Q$ be an interval retract, then the
M\"obius functions $\mu_P$ and $\mu_Q$ of $P$ and $Q$ are related by
the formula
\begin{equation}\label{F_eq: subGalois}
    \mu_Q(x,y) \ =\,
    \sum_{\substack{f(a)=x\\f(b)=y}} \mu_P(a,b) \qquad (\forall x,y\in Q).
\end{equation}
\end{theorem}
The proof in \cite{F_FLS:2010} relies on the fact that the
intersection of two intervals in a lattice is again an interval.

Here we define eight closely related maps in order to demonstrate
four new interval retracts: first four projections, each associated
to a corresponding section.

\[
\xymatrix
 {
&\Fm_n \ar@{->>}[dl]^{\gamma}\ar@{->>}[dr]_{\varphi}\\
\Fy_n\ar@{->>}[dr]_{\hat{\varphi}}&&\Fck_n\ar@{->>}[dl]^{\hat{\gamma}}\\
&\Fq_n } \quad\quad
\xymatrix
{
&\Fm_n \ar@{<-_{)}}[dl]^{\gamma_<}\ar@{<-^{)}}[dr]_{\varphi_>}\\
\Fy_n\ar@{<-^{)}}[dr]_{\hat{\varphi}_>}&&\Fck_n\ar@{<-_{)}}[dl]^{\hat{\gamma}_<}\\
&\Fq_n }
\]

The maps $\gamma$ and $\hat{\gamma}$ operate by replacing the
painted portion with a corolla, while $\varphi$ and $\hat{\varphi}$
replace the unpainted forest with a forest of corollas.

We define the sections $\gamma_<$ and $\hat{\gamma}_<$ by replacing painted
corollas with left combs, while $\varphi_>$ and $\hat{\varphi}_>$ are defined
by replacing unpainted corollas with right combs

The main result will be that each paired projection and section between the
same two lattices together define an interval retract. In
figure~\ref{F_f:mapmap} we demonstrate all the projections and bijections
described above.
\begin{figure}[hbt]
\[\includegraphics[width = \textwidth]{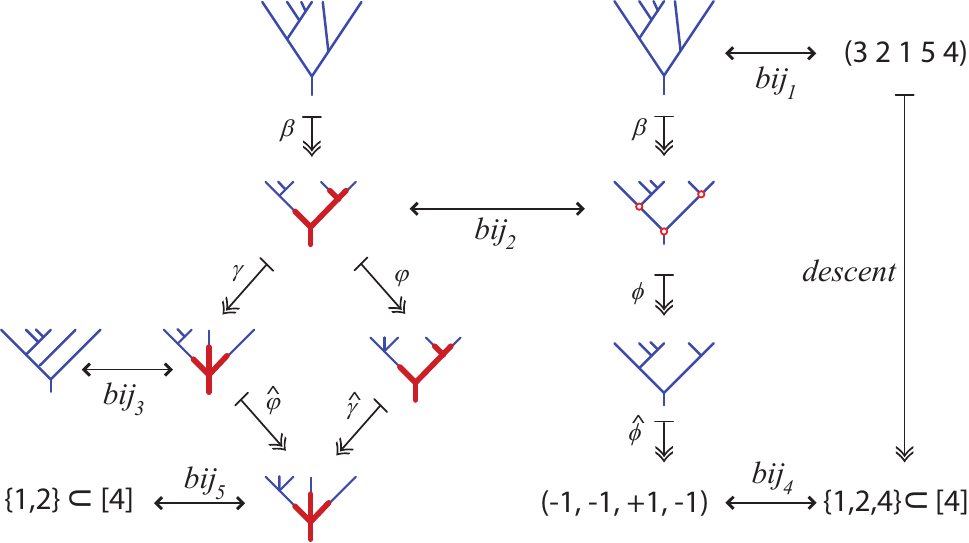}
\]
\caption{Projections and bijections in this paper. }
\label{F_f:mapmap}
\end{figure}
\begin{figure}[hbt]
\[\includegraphics[width = \textwidth]{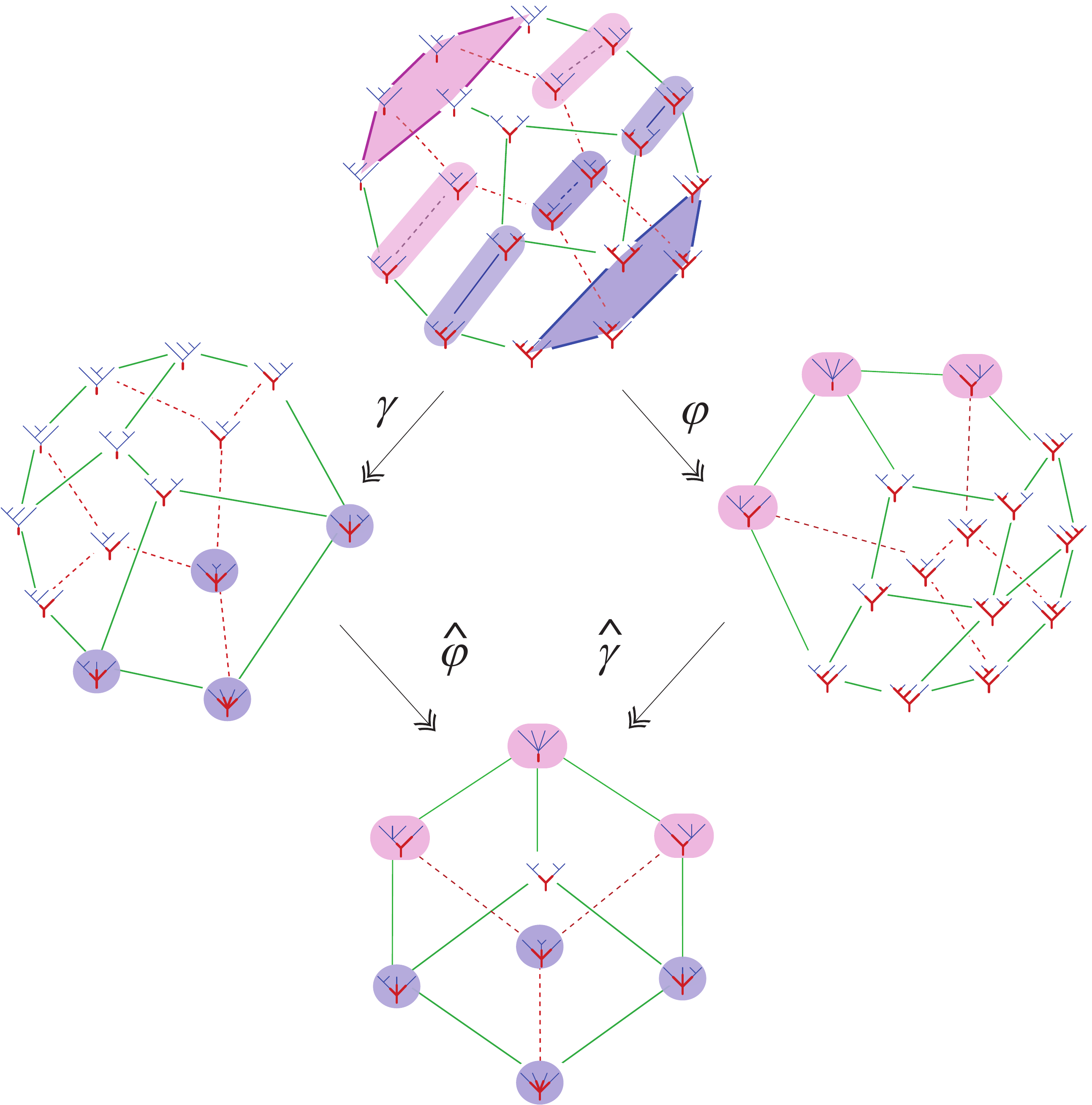}
\]
\caption{The four projections in dimension 3, with shaded intervals
retracted.} \label{F_f:multifiber}
\end{figure}
\begin{figure}[hbt]
\[\includegraphics[width = \textwidth]{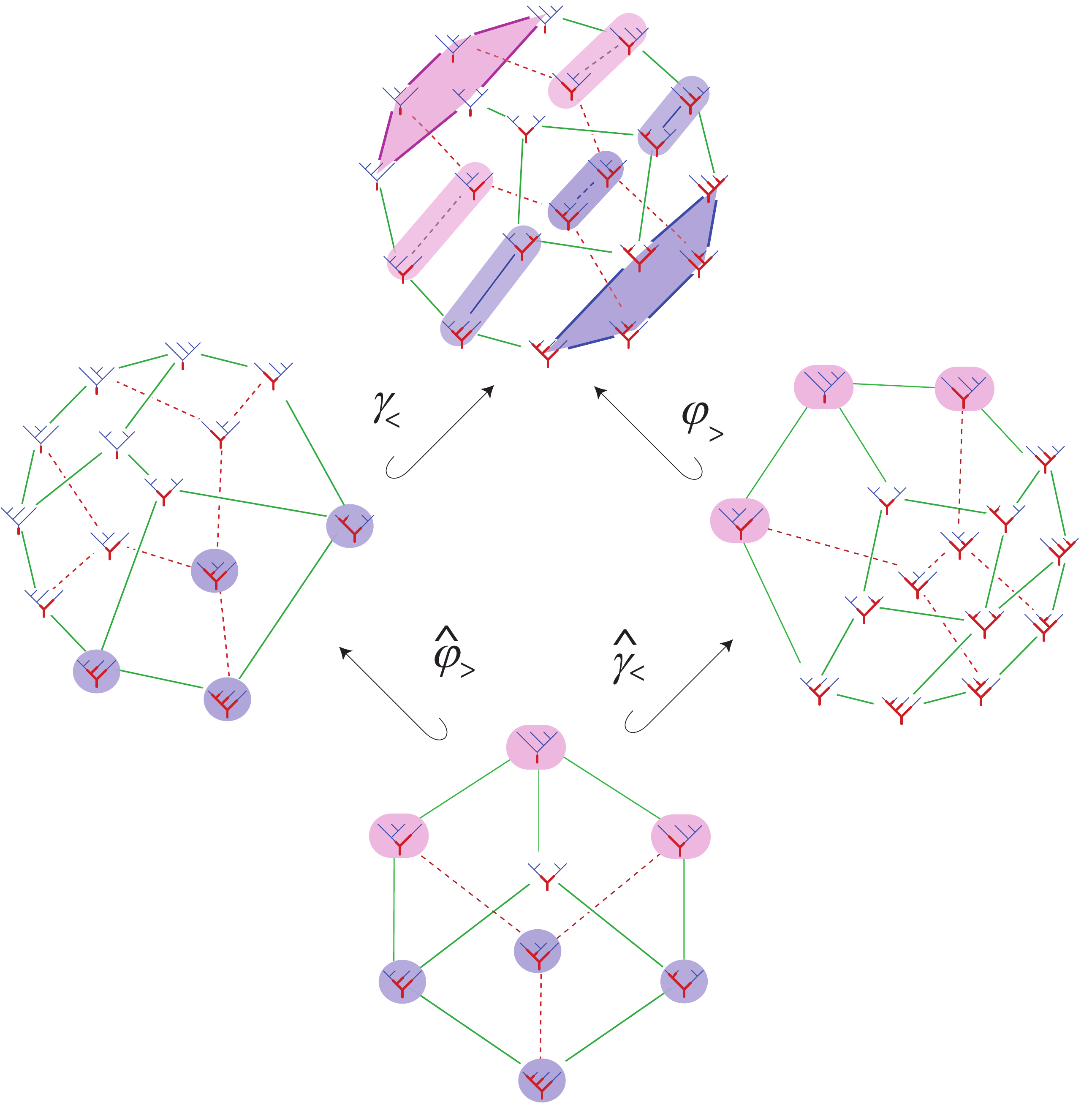}
\]
\caption{The four sections in dimension 3. Here we label the
elements of $\Fy_n,$ $\Fck_n$ and $\Fq_n$ using left and right combs
so that the sections can be seen as inclusions.}
\label{F_f:multisection}
\end{figure}

\begin{theorem}
The map $\gamma$  is an interval retract from $\Fm_n$ to $\Fy_n.$
\end{theorem}
\begin{proof}
We use the section $\gamma_<$. Thus we need to show four things:
that $\gamma^{-1}(t)$  is an interval in $\Fm_n$ for any $t\in
\Fy_n$; that $\gamma$ and $\gamma_<$ both preserve order, and that
$\gamma \circ \gamma_<$ is the identity map. This last fact is
straightforward, since it constitutes first removing and then
replacing an unpainted forest on its painted comb.

Consider $t\in \Fy_n$  with the $k$-forest $f$ of subtrees with
initial nodes on the left limb of $t.$  To show that inverse images
of $\gamma$ are intervals, we point out that $\gamma^{-1}(t)$ is the
set of painted trees  with unpainted forest $f$ and any painted
portion. This is an interval since its elements comprise all those
between a unique min and max given by minimizing and maximizing the
painted portion. In fact, the interval is isomorphic to a copy of
the Tamari lattice $\Fy_{k-1}.$

 Next to show that $\gamma$ preserves the order, we let
$a<b \in \Fm_n.$ This means that $t_a\le t_b$ in the Tamari order,
and that $P_a \supseteq P_b.$  Now we may visualize the action of
$\gamma$ as a series of smaller steps: first we make all possible
Tamari moves in the painted region of  $a$ that each yield
sequentially lesser trees. Then we attach a new branch to the
left-most painted  point of $a,$  and finally forget the painting
altogether. The same basic steps are performed to find $\gamma(b).$
Since $P_b \subseteq P_a$ we can see the relation $\gamma(b)\ge
\gamma(a)$ by the series of Tamari moves to get from $b$ to $a$
followed by more moves resulting from the possibly additional
painted nodes of $a$.

To show that if $q<s\in \Fy_n$ implies that $\gamma_<(q) \le
\gamma_<(s)$ , we consider  the string of Tamari covering moves that
relate $q$ to $s.$ Recall that $\gamma_<$ takes the $k$-forest $f$
of sub-trees attached along the left limb of $q$ and instead
attaches them to a minimal painted $k$-tree, that is, they are
grafted to a painted left comb with $k$ branches. Alternately this
is described by simply pruning away the leftmost leaf of $q$ and
painting the nodes of $q$ along the leftmost branch. We see that
$t_{\gamma_<(q)} \le t_{\gamma_<(s)}$ by noting that the moves
between them are the same as those from $q$ to $s.$ Then we note
that we have $P_{\gamma_<(q)} \supseteq P_{\gamma_<(s)}$ since any
move from $q$ to $s$ either subtracts from the set of painted nodes
in the eventual image (if the move involves a node on the leftmost
branch of $q$) or leaves that set unchanged.
\end{proof}

\begin{theorem}
The map $\varphi$  is an interval retract from $\Fm_n$ to $\Fck_n.$
\end{theorem}
\begin{proof}
We use the section $\varphi_>$. Again we need to show four things:
that $\varphi^{-1}(t)$  is an interval in $\Fm_n$ for any $w\in
\Fck_n$; that $\varphi$ and $\varphi_>$ both preserve order, and
that $\varphi\circ \varphi_>$ is the identity map. This last fact is
straightforward, since both maps will be the identity in this case
-- $\varphi_>$ will always be and $\varphi$ will be the identity
when applied to a painted tree with a forest of unpainted right
combs.

Recall that $\varphi$ involves replacing an unpainted forest $f$
with a forest of right combs. (Sometimes alternately drawn as
corollas or just a number.) Thus the fiber $\varphi^{-1}(w)$ for
$w\in \Fck_n$ is a collection of painted trees in $\Fm_n$ which
share the same set of painted nodes, and the same binary tree as the
subtree made up of those painted nodes -- but which may have any
forest of unpainted trees agreeing with those facts. Thus the fiber
is an interval bounded by choosing that forest to be all left or all
right combs. In fact  this is a cartesian product of associahedra.

Next to show that $\varphi$ preserves the order, we let $a<b \in
\Fm_n.$ This means that $t_a\le t_b$ in the Tamari order, and that
$P_a \supseteq P_b.$ First note that if $P_a = P_b,$ then
$\varphi(a)< \varphi(b)$ by Tamari moves in the painted nodes. If
$P_a \supset P_b,$ consider the forest of unpainted right combs of
$\varphi(b)$. Each of these combs has a leftmost node $k$. If the
corresponding node $k$ of $\varphi(a)$ is painted, then we can see
the relation as first performing Tamari moves on the right comb of
$\varphi(b)$ until we have  the binary tree supported by node $k$ of
$\varphi(a),$ and then allowing the paint level to rise to cover
node $k$ and any additional nodes to match $\varphi(a).$ These moves
performed on each unpainted comb of $\varphi(b)$ give us the result.

The section $\varphi_>$ is very simple; it merely returns us to the
maximum of the fiber. In fact, if we are using right combs for our
weighted trees then this section is the identity map, and so order
is clearly preserved.
\end{proof}

\begin{theorem}
The map $\hat{\varphi}$  is an interval retract from $\Fy_n$ to
$\Fq_n.$
\end{theorem}
\begin{proof}
It is easiest to see this when viewing $\Fy_n$ in its incarnation as
ordered forests of binary trees, grafted onto painted left combs.
Then the ordering of $\Fy_n$ is directly inherited from $\Fm_n,$ and
the map $\hat{\varphi}$ is the same as $\varphi.$ Thus the facts
that the fibers are intervals and that
$\hat{\varphi}\circ\hat{\varphi}_>$ is the identity are already
proven.

Now the elements of $\Fq_n$ are being drawn as composite trees
(using combs or corollas) but we need to check that the usual
ordering by inclusion of subsets (of unpainted nodes) agrees with
the tree order. That is, if $p<q$ as elements of $\Fy_n$ (each drawn
as an unpainted forest grafted to a left comb) then
$\hat{\varphi}(p)\le \hat{\varphi}(q).$ By viewing two elements of
$\Fq_n$ as forests of right combs grafted to painted left combs, we
see that the only relation inherited from $\Fm_n$ is that of the
containment of the sets of painted nodes. Thus since in $\Fm_n$ a
larger set of painted nodes is a lesser element, here a smaller set
of unpainted nodes is the lesser element.

Finally the section $\hat{\varphi}_>$ is given by inclusion of the
composite tree as a forest of right combs grafted to a painted left
comb, which ensures that the ordering is preserved.
\end{proof}

\begin{theorem}
The map $\hat{\gamma}$  is an interval retract from $\Fck_n$ to
$\Fq_n.$
\end{theorem}
\begin{proof}
We already view an element of $\Fck_n$ as a forest of right combs
grafted to a painted binary tree. Viewing $\hat{\gamma}$ as
replacing the painted nodes with a left comb, we see the proof
proceeds just as for $\hat{\varphi}.$
\end{proof}

\section{Lattices and polytopes.} 
Next we point out that the four interval retracts just defined
extend to well known cellular projections of the polytopes. These
projections are not the same ones that appear in the work of Reading
\cite{F_Rea:2005},  Loday and Ronco \cite{F_LodRon:1998}, or Tonks
\cite{F_Ton:1997}. Rather they are found implicitly in the work of
Boardman and Vogt on maps of homotopy $H$-spaces
\cite{F_BoaVog:1973}.

Recall that the combinatorial lattices of trees we have discussed
here all occur conveniently as the labels of vertices on convex
polytopes. The Hasse diagrams are specific drawings of the
1-skeleton of each polytope. The polytopes are associated to
another, entirely different, lattice: the poset of their faces, with
the empty set adjoined as least element. It is an open question as
far as I know whether there is any describable relationship between
the two lattices, for instance between the Tamari lattice and the
face-poset of the associahedron.

As indicated in the introduction (by our use of the same symbol for
both polytope and lattice) we have the following correspondence:
binary trees label vertices of the associahedra; ordered binary
trees the permutohedra; painted binary trees the multiplihedra,
weighted trees the composihedra; composition trees the hypercubes.
The higher dimensional faces of these polytopes are all associated
to further generalizations of the trees in question, by allowing
more non-binary nodes and by allowing painting to end precisely on a
node.

The Hopf algebras of binary trees, ordered trees and Boolean subsets
have all been extended to larger Hopf algebras on the faces of the
corresponding polytopes. This was achieved by Chapoton in
\cite{F_Cha:2000}. It is the topic of future study that similar
expansions exist for the multiplihedra and composihedra.

Here we show how the projections discussed in this paper appear as
collapsing of the faces of the polytopes whose vertices they act
upon. Figure~\ref{F_fig:_multi_above_perm} shows an alternate
``above'' view of the permutohedron. This is given in order to
facilitate contrasting the various projection maps.
\begin{figure}[htb]
\[
    \includegraphics[width = 2.5in]{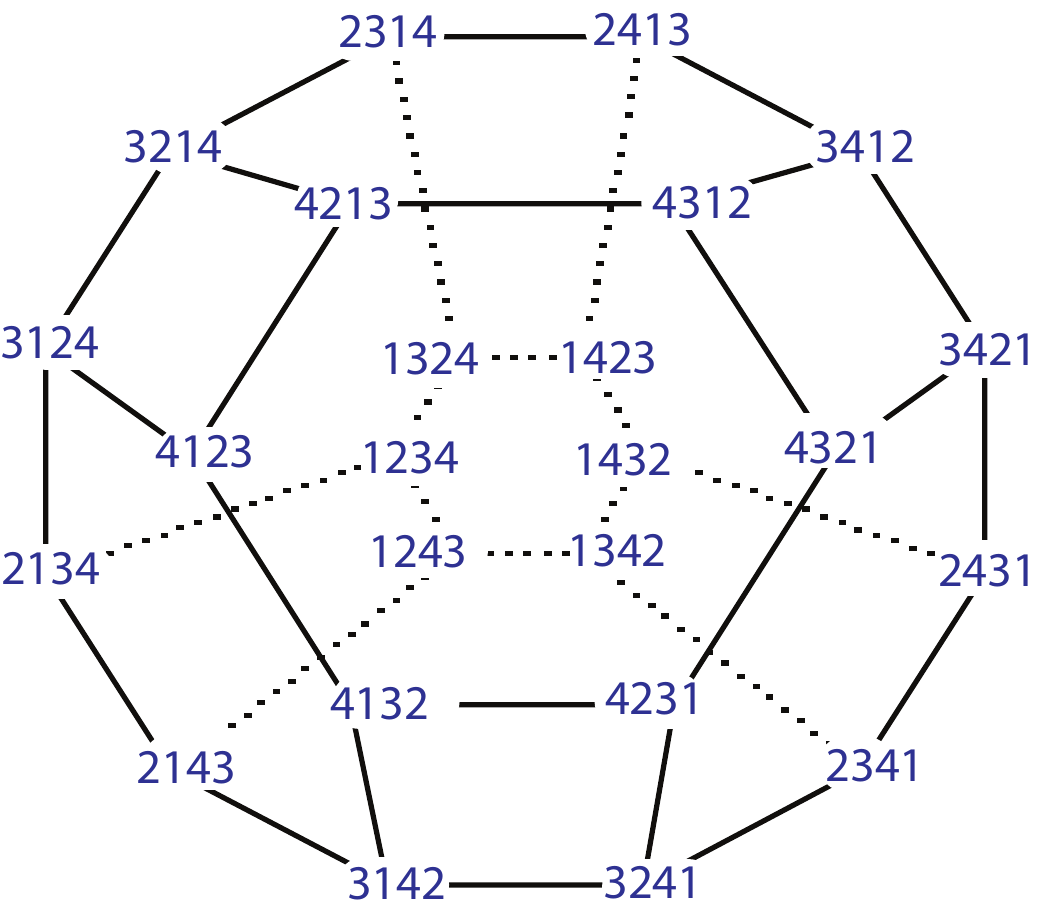}
\]
\caption{The 3d permutohedron, alternate view.}
\label{F_fig:_multi_above_perm}
\end{figure}

 Figure~\ref{F_fig:_multi_collapse_compare} offers contrast and comparison of our new maps to the classic projections, showing the faces that
are retracted. We use the ``above'' view of the posets. In the pictured 3d case
it is apparent that the two compositions of maps,
$\hat{\phi}\circ\phi\circ\beta$ (which is the Tonks projection) and
$\hat{\varphi}\circ\gamma\circ\beta$, have quite different actions on the
permutohedron. The number of collapsed cells are the same in both composite
projections, but in the first the image of the collapsed cells (4 hexagons and
4 rectangles) is a copy of $S^1$ where in the second the image (of 2 hexagons
and 6 rectangles) consists of two disjoint star graphs.

Next, for comparison, is the Tonks projection again, factored through the
cyclohedron as in \cite{F_ForSpr:2010}. Finally for further contrast we include
the projection $\eta$ defined by Reading in his theory of Cambrian lattices
\cite{F_Read:now}. 
\begin{figure}[htb]
\[
    \includegraphics[width = \textwidth]{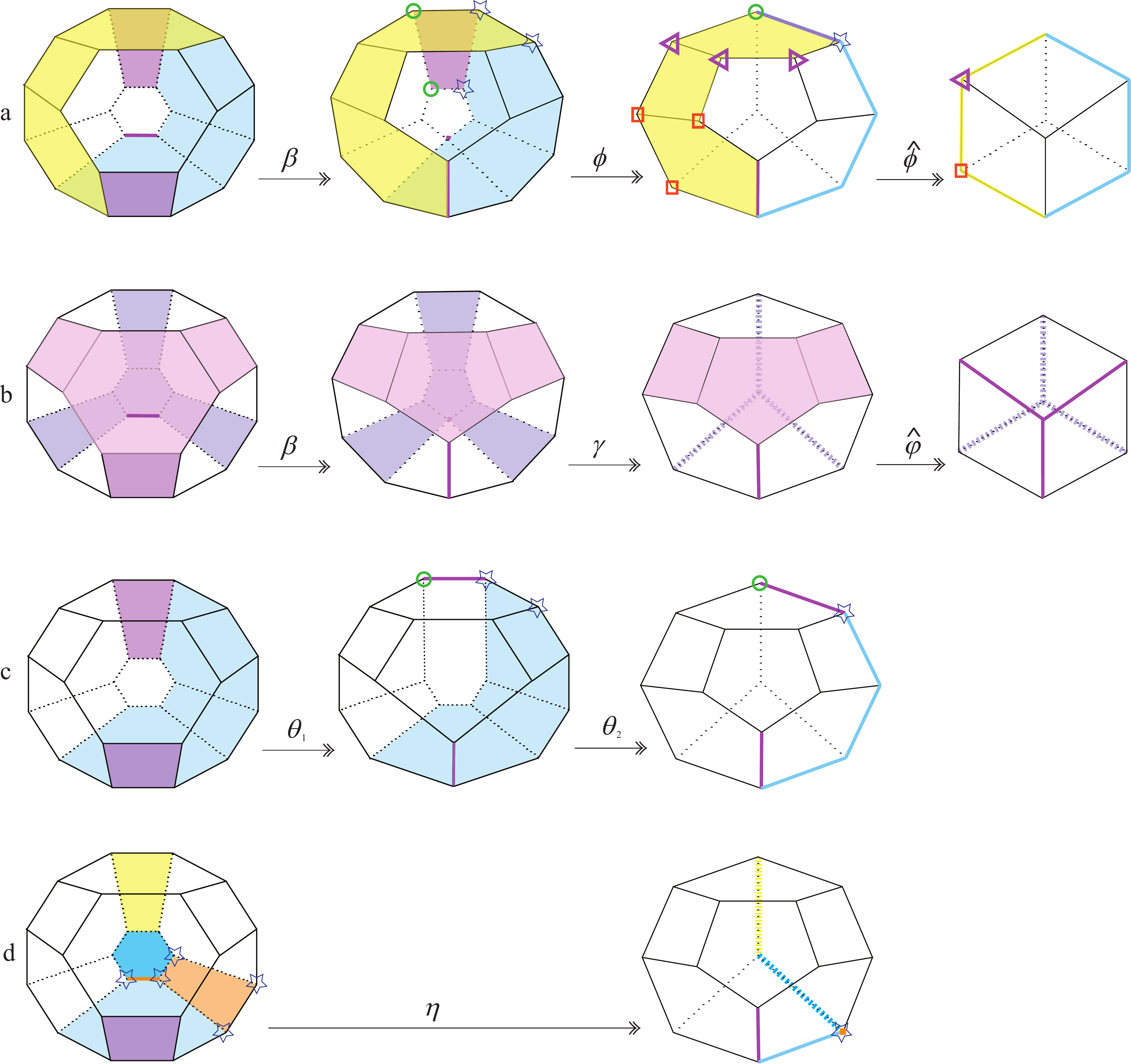}
\]
\caption{In a, b, c and d the permutohedron is oriented as in
Figure~\ref{F_fig:_multi_above_perm}. The shaded facets and edges
are collapsed in succession to similarly shaded edges and points. In
a the circled and starred vertices in the domain of $\phi$ are
mapped one and all to the circled (respectively starred) vertex in
the range, and likewise for the squared and triangled vertices in
the domain and range of $\hat{\phi}$. In b both $\gamma$ and
$\hat{\varphi}$ collapse a pentagon in their respective domains to a
single vertex in their respective ranges. In c, the central polytope
is the 3d cyclohedron. Finally in d $\eta$ takes all the starred
vertices to a single vertex.} \label{F_fig:_multi_collapse_compare}
\end{figure}

The factorization of the Tonks projection, $\Theta =
\theta_2\circ\theta_1,$ through the cyclohedron seen in
Figure~\ref{F_fig:_multi_collapse_compare} deserves some special
mention. First, this factorization is defined in greater generality
\cite{F_ForSpr:2010} in terms of \emph{tubings} of simple graphs.

\begin{definition}
Let $G$ be a finite connected simple graph, with $n$ numbered nodes.
A \emph{tube} is a set of nodes of $G$ whose induced graph is a
connected subgraph of $G$. Two tubes $u$ and $v$ may interact on the
graph as follows:
\begin{enumerate}
\item Tubes are \emph{nested} if  $u \subset v$.
\item Tubes are \emph{far apart} if $u \cup v$ is not a tube in $G,$ that is, the induced subgraph of the union
 is not connected, or
none of the nodes of $u$ are adjacent to a node of $v$.
\end{enumerate}
Tubes are \emph{compatible} if they are either nested or far apart.
We call $G$ itself the \emph{universal tube}.
 A \emph{tubing} $T$ of $G$ is a set of tubes of $G$ such that every pair of tubes in $T$ is
 compatible; moreover, we force every tubing of $G$ to contain (by default) its universal tube.
By the
 term $k$-\emph{tubing} we refer to a tubing made up of $k$ tubes, for $k \in \{1,\dots,n\}.$
 \end{definition}

 \begin{theorem} {\textup{\cite[Section 3]{F_dev-carr}}}
For a graph $G$ with $n$ nodes, the \emph{graph associahedron}
${\mathcal K} G$ is a simple, convex polytope of dimension $n-1$
whose face poset is isomorphic to the set of tubings of $G$, ordered
such that $T \prec T'$ if $T$ is
 obtained from $T'$ by adding tubes.
\label{d:pg}
\end{theorem}
 The vertices of the graph associahedron are the $n$-tubings of $G.$
 Faces of dimension $k$ are indexed by $(n-k)$-tubings of $G.$

As seen in \cite{F_dev-carr}, the permutohedron $\Fs_n ={\mathcal K}
G$ where $G$ is the complete graph on $n$ nodes; the associahedron
$\Fy_n ={\mathcal K} G$ where $G$ is the path graph on $n$ nodes;
the cyclohedron is $\mathcal W_n = {\mathcal K} G$ when $G$ is the
cycle on $n$ nodes; and the stellohedron is ${\mathcal K} G$ when
$G$ is the star graph on $n$ nodes. 

The question might be asked: how easily may the weak order on
permutations and the Tamari order be generalized to $n$-tubings on a
graph with nodes numbered $1,...,n$?
 In order to describe the ordering we give the covering relations.
 We can use the same notation as when comparing tubings in the poset
 of faces of the graph associahedron since in that poset the
 $n$-tubings are not comparable.

\begin{definition} Two $n$-tubings $T, T'$ are in a covering relation $T'\prec T$ if they have
all the same tubes except for one differing pair. We actually
compare the outermost nodes, one from each of the pair of differing
tubes. The outermost node of a tube is the node that is included in
no other smaller sub-tube of the tubing. If the number of that node
is greater for $T$, then $T$ covers $T'.$
\end{definition}
Note that each such covering relation corresponds to a unique
$(n-1)$-tubing: the one resulting from removing the differing tubes.
Thus the covering relations correspond to the edges of the graph
associahedron.

For example, in Figure~\ref{F_fig:tube_perm} we show a covering relation
between two tubings on the complete graph on four numbered nodes. This figure
also demonstrates the bijection between $n$-tubings and permutations of $[n].$
The nodes are the inputs for the permutation, and the output is the relative
tube size. E.g., in the left-hand permutation the image of 2 is 1, and so we
put the smallest tube around 2. To see the relation via tubes, we write down
the sets of nodes in each tube. Only one pair of tubes differs.  We compare the
two numbered nodes of these which are in no smaller tubes. Here $(3 1 2 4)
\prec (4 1 2 3)$  since $1<4.$
\begin{figure}[htb]
\[
    \includegraphics[width=\textwidth]{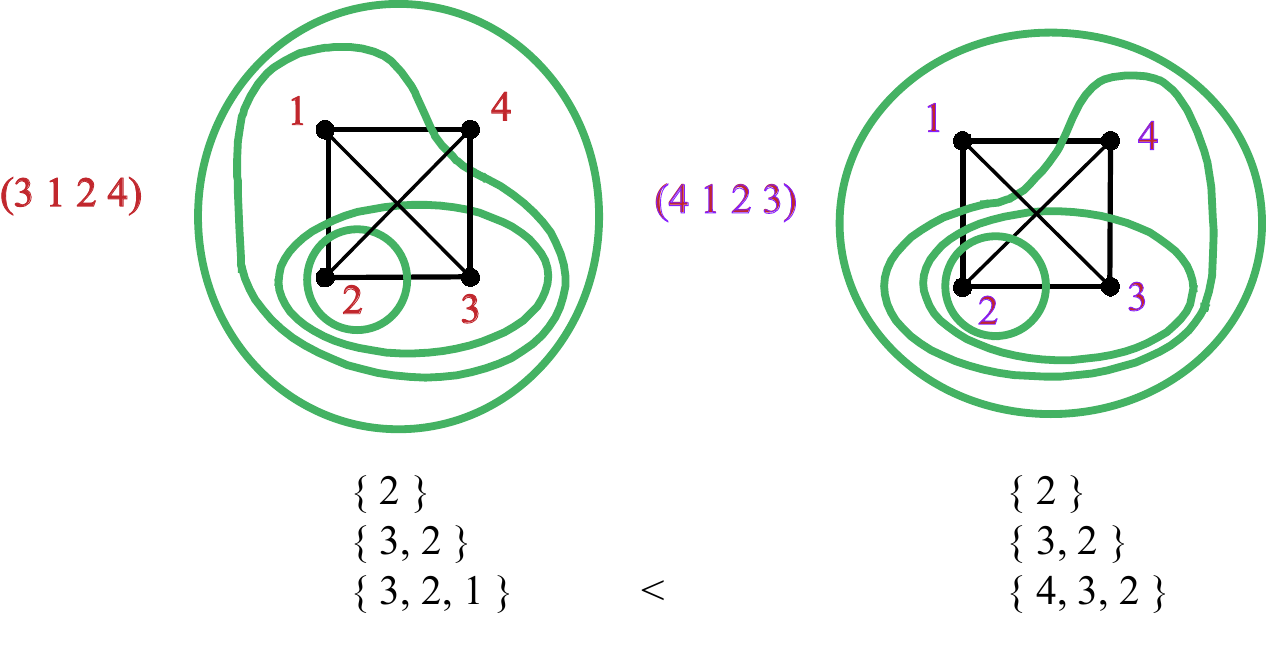}
\]
\caption{A covering relation in the weak order on
permutations.}\label{F_fig:tube_perm}
\end{figure}

It turns out that the relation generated by these covering relations of tubings
has been independently demonstrated to be a poset by Ronco \cite{F_Ronco:now}.
In her article, the poset we have described on $n$-tubings of a graph is seen
as the restriction of a larger poset on all the tubings of a graph.

Figure~\ref{F_fig:tube_cycle} shows  the lattice that results from
the cycle graph, rocovering the cyclohedron in dimension 3.
 The Hasse diagram is combinatorially
equivalent to the 1-skeleton of the cyclohedron. Notice that this is quite
different from the type $B_3$ Cambrian lattice described by Reading in this
volume \cite{F_Read:now}, despite the fact that the latter also is
combinatorially equivalent to the 1-skeleton of the cyclohedron.
Figure~\ref{F_fig:tube_star} shows the corresponding lattice on $4$-tubings of
the star graph on 4 nodes. This Hasse diagram is combinatorially equivalent to
the 1-skeleton of the 3d stellohedron. Figure~\ref{F_fig:tube_bare} shows both
the cyclohedron and stellohedron lattices again, unlabeled, with a different
view of each polytope for comparison.
\begin{figure}[htb]
\[
    \includegraphics[width=\textwidth]{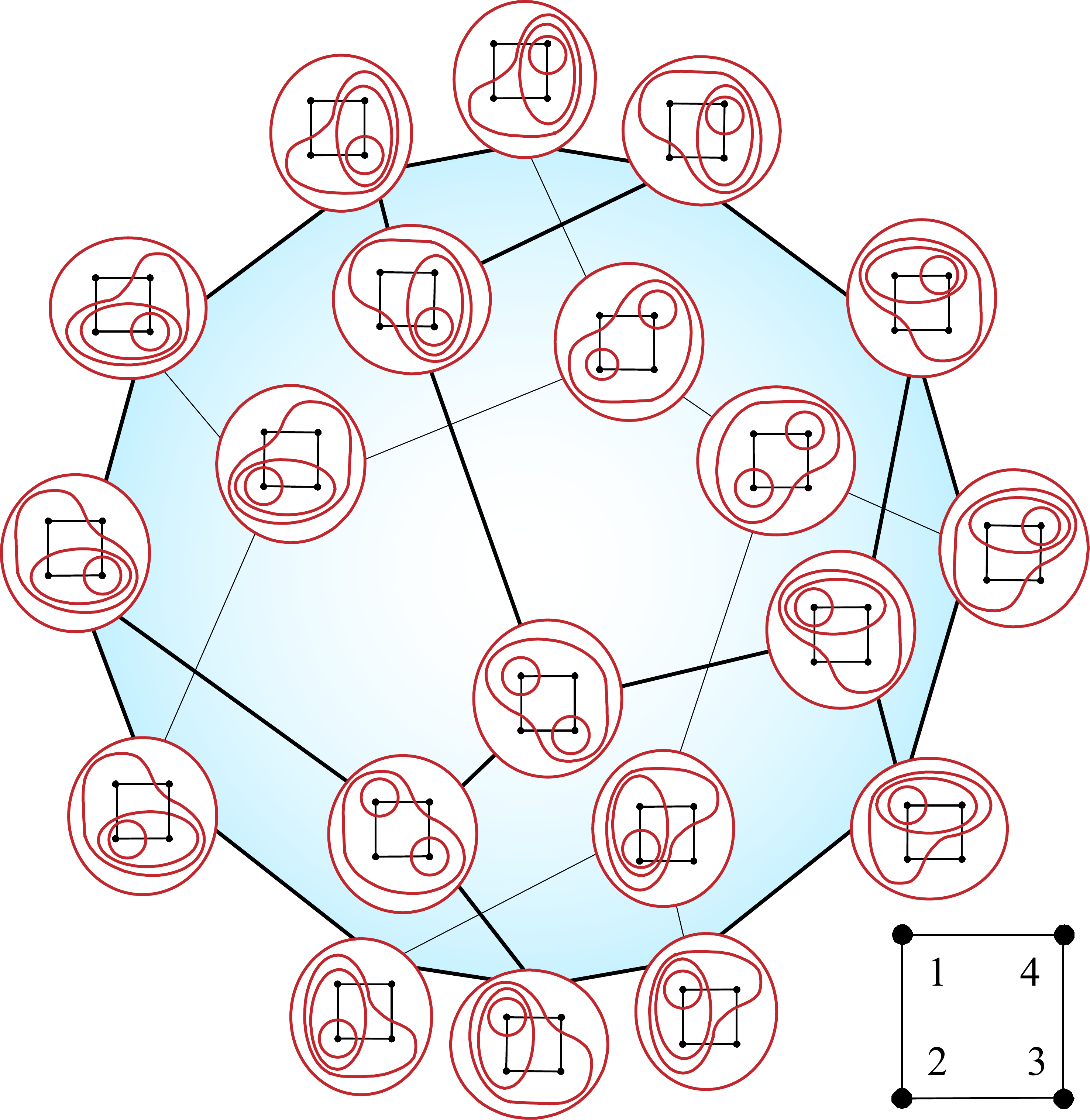}
\]
\caption{This Hasse diagram is labeled by tubings of the cycle
graph, with nodes numbered 1--4. The covering relations are also a
picture of the edges of the cyclohedron.}\label{F_fig:tube_cycle}
\end{figure}
\begin{figure}[htb]
\[
    \includegraphics[width=\textwidth]{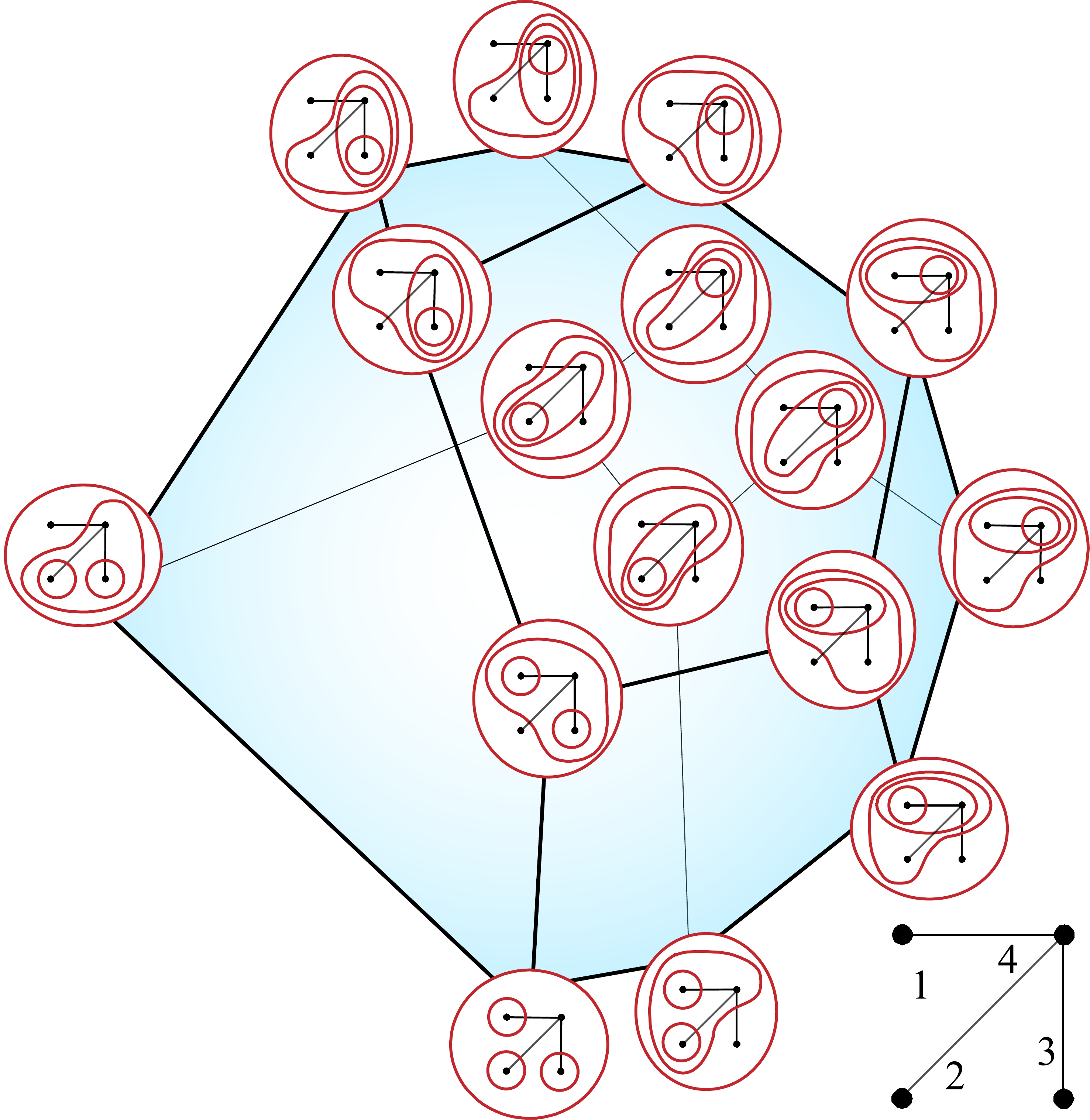}
\]
\caption{This Hasse diagram is labeled by tubings of the star graph,
with nodes numbered 1--4. The covering relations are also a picture
of the edges of the stellohedron.}\label{F_fig:tube_star}
\end{figure}
\begin{figure}[htb]
\[
  \includegraphics[width=\textwidth]{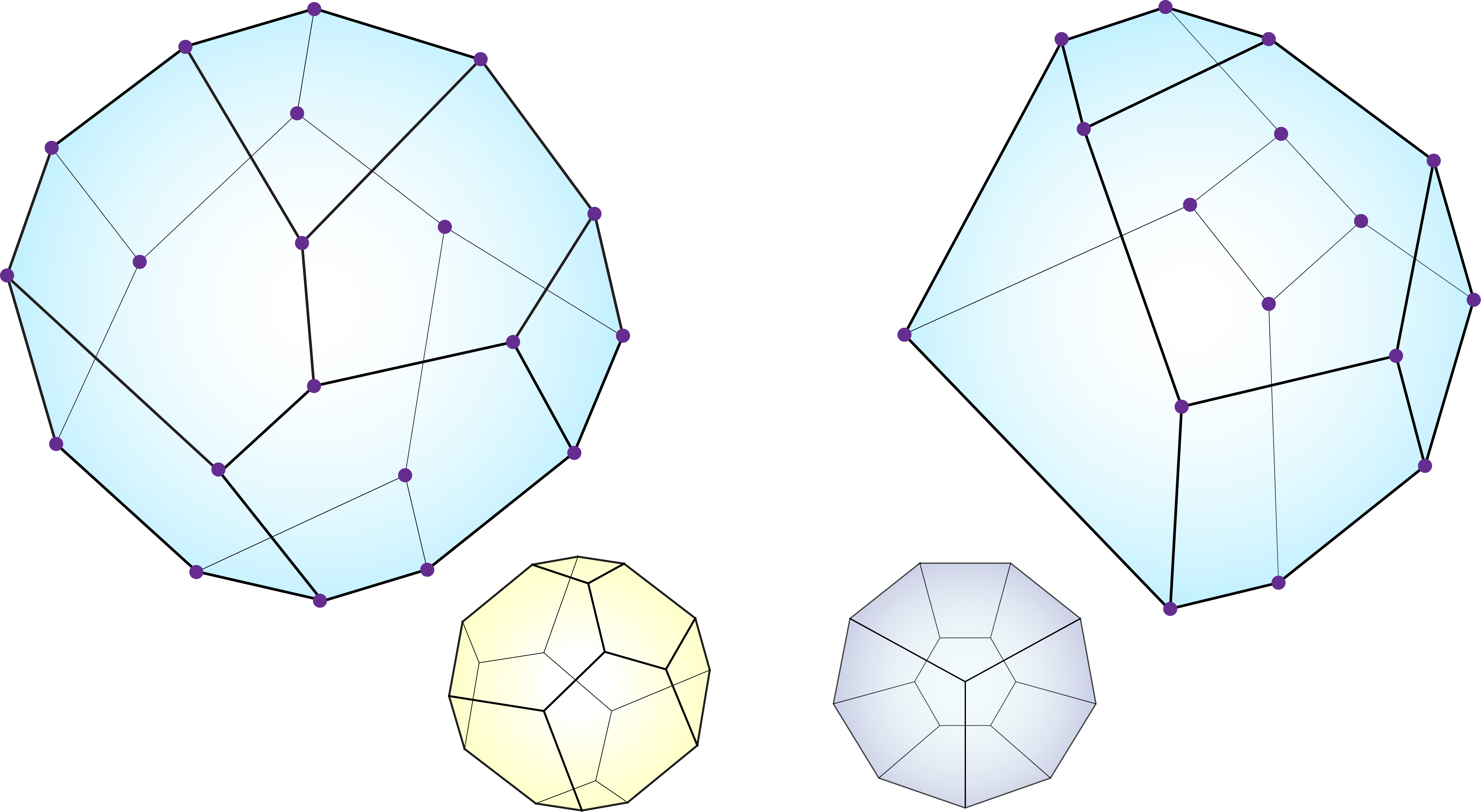}
 \]
\caption{On the left is the Hasse diagram for $n$-tubings of the
cycle graph, with the cyclohedron pictured below for comparison. On
the right is the Hasse diagram for $n$-tubings of the star graph,
with the stellohedron pictured below.}\label{F_fig:tube_bare}
\end{figure}

We note that as seen in Ronco's article \cite{F_Ronco:now}, the
Tamari lattice is found as the lattice of $n$-tubings on the path
graph with nodes numbered $1,\dots,n$ in the order that they are
connected by edges.  Several open questions present themselves: for
one, we notice that the 3-dimensional graph associahedra pictured
here have associated posets which upon inspection prove to be
lattices--it is not clear that they always are.

\section{Algebraic implications of interval retracts} 
Finally we point out the importance of these lattices to the Hopf
algebras defined as spans of their elements. For each of the
lattices studied here, there is a graded vector space given by the
direct product of the spans of the vertices of the $n$-dimensional
polytope. For instance a vector space of binary trees is defined as:
$$ \Fysym = \bigoplus_{n\geq0} span~{\Fy_n} $$
The binary trees index a basis, called the fundamental basis and
denoted $\{F_t~|~t\in \Fy_n\}.$ There is a graded Hopf algebra
structure on this vector space, well studied in
\cite{F_AguSot:2006}. Similarly
$$ \Fssym = \bigoplus_{n\geq0} span~{\Fs_n} $$
 is a graded Hopf algebra on ordered trees, well studied in
 \cite{F_AguSot:2005}.

 Here we will restrict our attention to the coalgebra structures,
 which interact in important ways with the lattice structures.
The remainder of this section is taken in part from \cite{F_FLS3}. It is
included in order to demonstrate the algebraic importance of the lattice
structures.

 \subsection{Coalgebras of trees} 
 We
define \emph{splitting} a binary tree $w$ along the path from a leaf to the
root to yield a pair of binary trees,
\[\includegraphics{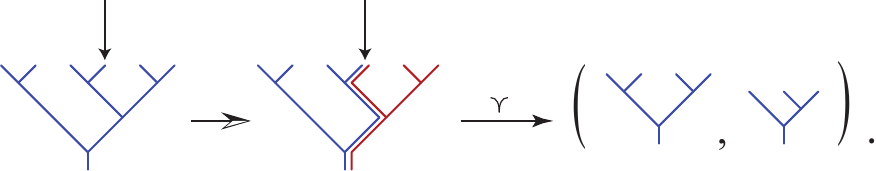}
\]
Write $w\Fpsplit(w_0,w_1)$ when the  pair of  trees $(w_0,w_1)$ is
obtained by splitting $w$.

 \begin{definition}[Coproduct on $\Fysym$]\label{F_def: ysym}
 Given a binary tree $t$, define the coproduct in the fundamental
 basis by
 \begin{gather*}
        \Delta(F_t)\ =\ \sum_{t \Fpsplit (t_0,t_1)} F_{t_0} \otimes F_{t_1}
        .
 \end{gather*}
 \end{definition}
Here is an example:\newline
\[\includegraphics{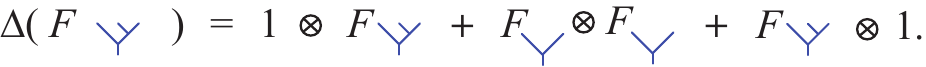}
\]

\subsection{Cofree composition of coalgebras}\label{F_sec: cccc main results}
The following is excerpted with edits from \cite{F_FLS3}. Let
${\mathcal{C}}$ and ${\mathcal{D}}$ be graded coalgebras. We form a
new coalgebra ${\mathcal{E}}={\mathcal{D}}\circ{\mathcal{C}}$ on the
vector space
 \begin{gather}\label{F_eq: E_(n)}
  {\mathcal{D}}\circ{\mathcal{C}}\ := \ \bigoplus_{n\geq0} {\mathcal{D}}_n \otimes {\mathcal{C}}^{\otimes(n+1)} \,.
 \end{gather}
We write
${\mathcal{E}}=\bigoplus_{n\geq0}{\mathcal{E}}_{\FDdeg{n}}$, where
${\mathcal{E}}_{\FDdeg{n}}= {\mathcal{D}}_n \otimes
{\mathcal{C}}^{\otimes(n+1)}$. This gives a coarse coalgebra grading
of ${\mathcal{E}}$ by \emph{${\mathcal{D}}$-degree}. There is a
finer grading of ${\mathcal{E}}$ by \emph{total degree}, in which a
decomposable tensor $c_0\otimes \dotsb\otimes c_n \otimes d$ (with
$d\in{\mathcal{D}}_n$) has total degree $|c_0|+\dotsb+|c_n|+|d|$.
Write ${\mathcal{E}}_n$ for the linear span of elements of total
degree $n$.

\begin{example}\label{F_ex: painted}
 This composition is motivated by a grafting construction on trees.  Let
 $d\times (c_0,\dotsc,c_n) \in {\Fy_n} \times \bigl({\Fy_\Fbb}^{{n+1}}\bigr)$.
 Define $\circ$ by attaching the forest $(c_0,\dots,c_n)$ to the leaves of $d$
 while remembering $d$,
 \[
   \includegraphics{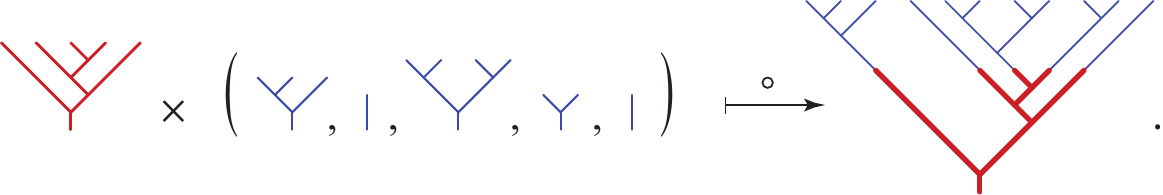}
 \]
 This is precisely the type of tree called a \emph{painted tree} in  Section 2.
 Applying this construction to the indices of basis elements of ${\mathcal{C}}$ and ${\mathcal{D}}$ and
 extending by multilinearity gives ${\mathcal{C}} \circ {\mathcal{D}}$.

\end{example}

Motivated by this example, we represent an decomposable tensor in
${\mathcal{D}}\circ{\mathcal{C}}$ as
\[
    \Fcomposeinline{c_0}{c_n}{d} \qquad\hbox{or}\qquad \Fcompose{c_0}{c_n}{d}
\]
to compactify notation.

\subsection{The coalgebra of painted trees.}

Let $\Fm_n$ be the poset of painted trees on $n$ internal nodes.
Then the vector space $\Fpsym = \Fysym \circ \Fysym$ may be directly
given by:
$$ \Fpsym = \bigoplus_{n\geq0} span~{\Fm_n} $$
We reproduce the compositional coproduct defined in Section 2 of
\cite{F_FLS3}.

\begin{definition}[Coproduct on $\Fpsym$]\label{F_def: painted is cccc}
 Given a painted tree $p$, define the coproduct in the fundamental
 basis $\bigl\{ F_p \mid p \in \Fm_\Fbb \bigr\}$ by
 \begin{gather*}
        \Delta(F_p)\ =\ \sum_{p \Fpsplit (p_0,p_1)} F_{p_0} \otimes F_{p_1} ,
 \end{gather*}
 where the painting in $p$ is preserved in the splitting $p\Fpsplit(p_0,p_1)$.
\end{definition}

The counit $\varepsilon$ satisfies $\varepsilon(F_p) =
\delta_{0,|p|}$, the Kronecker delta, as usual for graded
coalgebras.

\subsubsection{Primitives in the coalgebras of trees and painted
trees}\label{F_sec: painted primitives} Now for the discussion of
how the lattice structure found by Tamari really impacts the
algebraic structure. Recall that a primitive element $x$ of a
coalgebra is such that $\Delta x = 1\otimes x + x\otimes 1.$ Theorem
2.4 of \cite{F_FLS3} describes the primitive elements of $\Fpsym =
\Fysym \circ \Fysym$ in terms of the primitive elements of $\Fysym$.
We recall the description of primitive elements of $\Fysym$ as given
in~\cite{F_AguSot:2006}. 

Let $\mu$ be the M\"obius function of $\Fy_n$ which is defined by
$\mu(t,s)=0$ unless $t\leq s$,
\[
   \mu(t,t)\ =\ 1\,, \qquad\mbox{and}\qquad
   \mu(t,r)\ =\ -\sum_{t\leq s< r} \mu(t,s)\,.
\]
We define a new basis for $\Fysym$ using the M\"obius function. For
$t\in\Fy_n$, set
\[
   M_t\ :=\ \sum_{t\leq s} \mu(t,s) F_s\,.
\]
Then the coproduct for $\Fysym$ with respect to this $M$-basis is
still given by splitting of trees, but only at leaves emanating
directly from the right limb above the root:
\[\includegraphics{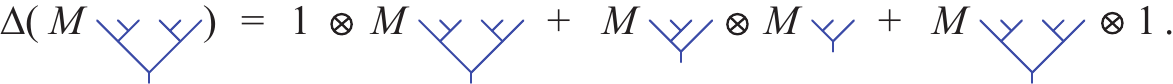}
\]
A tree $t\in\Fy_n$ is \emph{progressive} if it has no branching
along the right branch above the root node. A consequence of the
description of the coproduct in this $M$-basis is Corollary 5.3
of~\cite{F_AguSot:2006} that the set $\{M_t \mid t \hbox{ is
progressive}\}$ is a linear basis for the space of primitive
elements in $\Fysym$.

Now according to Theorem 2.4 of \cite{F_FLS3} the cogenerating
primitives in $\Fpsym$ are of two types:
\[
  \Flrcompose{1}{c_1}{c_{n-1}}{1}{M_t} \qquad\hbox{ and }\qquad \Fzcompose{\,M_t\,}{1} \,,
\]
where $t$ is a progressive tree. Figure~\ref{F_types} shows
examples.
\begin{figure}[htb]
\[ \includegraphics{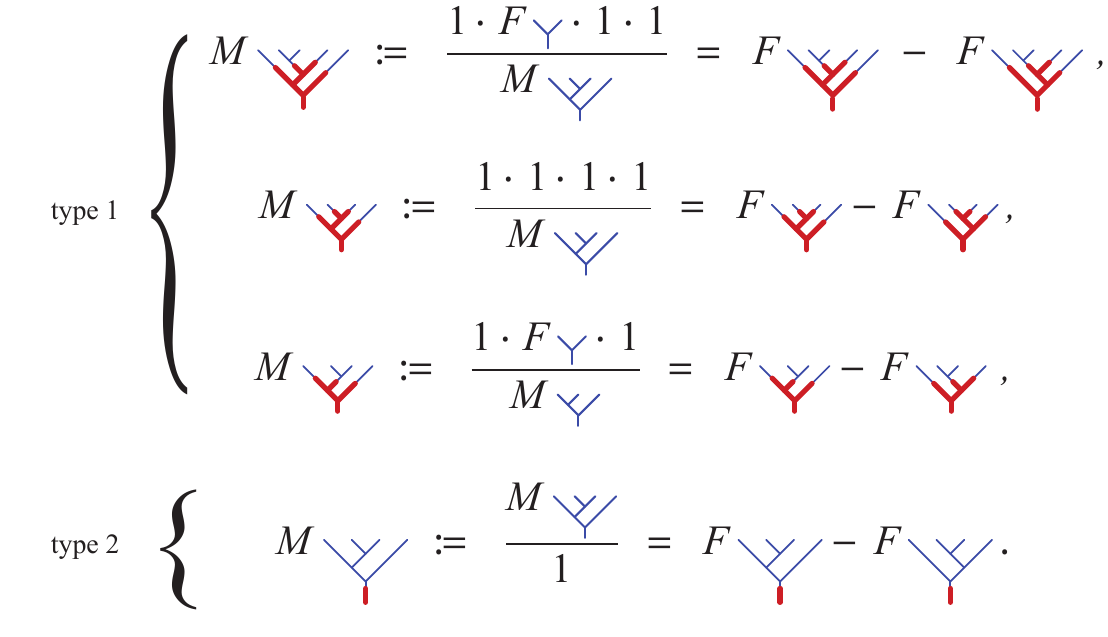}
\]
\caption{Primitive elements of two types in $\Fpsym.$}
\label{F_types}
\end{figure}

The primitives can be described in terms of M\"obius inversion on
certain subintervals of the multiplihedra lattice. For primitives of
the first type, the subintervals are those with a fixed unpainted
forest of the form $(~|\Ftreeseparator t \Ftreeseparator
   \dotsb\Ftreeseparator s\Ftreeseparator |~)$.
For primitives of the second type, the subinterval consists of those
trees whose painted part is trivial, i.e. only the root is painted.
Each subinterval of the first type is isomorphic to $\Fy_m$ for some
$m\leq n$, and the second subinterval is isomorphic to $\Fy_{n}$.
Figure~\ref{F_fig:_multi_sub} shows the multiplihedron lattice for
$\Fm_4$, with these subintervals highlighted.
\begin{figure}[htb]
\[
  \begin{picture}(260,238)(2,2)
   \put(2,5){\includegraphics[height=240pt]{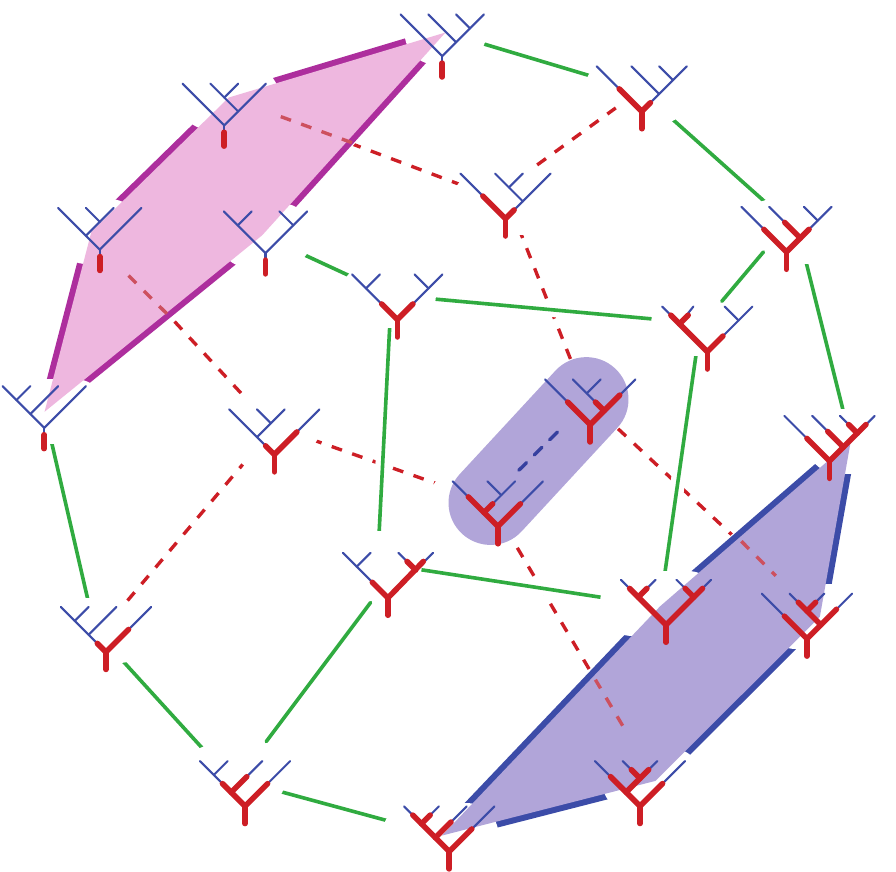}}
   \end{picture}
\]
\caption{The multiplihedron lattice $\Fm_4$ showing the three
subintervals that yield primitives via M\"obius transformation.}
\label{F_fig:_multi_sub}
\end{figure}

\bibliographystyle{amsplain}

\bibliography{F_bibl}
\label{F_sec:biblio}
\end{document}